\documentclass[12pt,reqno]{amsart}
\usepackage[dvips]{color}
\usepackage{amsmath}
\usepackage{amsxtra}
\usepackage{amscd}
\usepackage{amsthm}
\usepackage{amsfonts}
\usepackage{amssymb,hyperref}
\usepackage{eucal}
\usepackage{epsfig}
\usepackage{graphics}
\usepackage[matrix,arrow]{xy}

\input xy
\xyoption{all}

\theoremstyle{plain}
\newtheorem{theorem}{Theorem}
\newtheorem{lemma}[theorem]{Lemma}
\newtheorem{thm}[theorem]{Theorem}
\newtheorem{cor}[theorem]{Corollary}
\newtheorem{prop}[theorem]{Proposition}

\newtheorem{example}[theorem]{Example}
\newtheorem{exercise}[theorem]{Exercise}

\newcommand{\g}{\mathfrak{g}}

\newcommand{\nc}{\newcommand}
\nc{\on}{\operatorname}
\nc{\la}{\lambda}
\nc{\wh}{\widehat}
\nc{\wt}{\widetilde}
\nc{\sw}{{\mathfrak s}{\mathfrak l}}
\nc{\ghat}{\wh{\g}}
\nc{\hhat}{\wh{\h}}
\nc{\mc}{\mathcal}
\nc{\bi}{\bibitem}
\nc{\pa}{\partial}
\nc{\ppart}{(\!(t)\!)}
\nc{\pparx}{(\!(X)\!)}
\nc{\zpart}{(\!(z)\!)}
\nc{\n}{{\mathfrak n}}
\nc{\ol}{\overline}
\nc{\mb}{\mathbf}
\nc{\bb}{{\mathfrak b}}
\nc{\su}{\wh\sw_2}
\nc{\h}{{\mathfrak h}}
\nc{\can}{\on{can}}
\nc{\ntil}{\wt{\n}}
\nc{\pone}{{\mathbb P}^1}
\nc{\bs}{\backslash}
\nc{\al}{\alpha}
\nc{\gt}{{\mathfrak g}'}
\nc{\ds}{\displaystyle}
\nc{\Bun}{\on{Bun}}
\nc{\Gr}{\on{Gr}}

\def\neg{\negthinspace}

\nc{\ka}{\kappa}

\nc{\lka}{{}^L\neg\ka}
\nc{\LP}{{}^L\neg P}
\nc{\LQ}{{}^L\neg Q}
\nc{\lkt}{{}^L\neg\wt\ka}
\nc{\Hom}{\on{Hom}}
\nc{\OO}{\mathcal O}
\nc{\Loc}{\on{Loc}}
\nc{\sto}{\!\!\shortto\!\!}
\nc{\hsl}{\widehat{\mathfrak sl}}

\nc{\Fq}{{\mathbb F}_q}
\nc{\Fone}{{\mathbb F}_1}

\nc{\Pic}{\on{Pic}}

\theoremstyle{definition}

\usepackage{stmaryrd}
\usepackage{trimclip}

\makeatletter
\DeclareRobustCommand{\shortto}{%
  \mathrel{\mathpalette\short@to\relax}%
}

\newcommand{\short@to}[2]{%
  \mkern2mu
  \clipbox{{.3\width} 0 0 0}{$\m@th#1\vphantom{+}{\shortrightarrow}$}%
  }
\makeatother
\begin{document}

\title{Iterating sine, equivalence classes of variable changes, and groups with few conjugacy classes}

\author{Pavel Etingof}

\address{Department of Mathematics, MIT, Cambridge, MA 02139, USA}

\maketitle

\begin{abstract}
This is an expository paper about iterations of 
a smooth real function $f$ on $[0,\varepsilon)$ 
such that $f(0)=0$, $f'(0)=1$, and $f(x)<x$ for $x>0$, i.e., the sequence defined by $x_{n+1}=f(x_n)$. This sequence has interesting asymptotics, whose study leads to the question of classifying conjugacy classes in the group of formal changes of variable $y=f(x)$, i.e., formal series $f(x)=x+a_2x^2+a_3x^2+...$ with real coefficients (under composition). The same classification applies over a finite field $\Bbb F_p$ for suitably truncated series $f$, defining a family of $p$-groups which have the smallest number of conjugacy classes for a given order, i.e., are the ``most noncommutative" finite groups currently known. The paper should be accessible to undergraduates and at least partially to advanced high school students. 
\end{abstract}

\tableofcontents

\section{Iterations of $\sin x$: numerical experiments} What happens if we begin with some number $0<x_{0}<\pi$ and start pressing the ${\bf sin}$ button on the calculator, i.e., compute terms of the sequence defined by the recursion
$$
x_{n+1}=\sin x_{n}?
$$
How does this sequence behave?

First, $0<x_{n+1}<x_{n}$, so $x_{n}$ is decreasing and positive, thus has a limit $x$. This limit must satisfy the equation $x=\sin x$, so $x=0$. Thus 
$$
\lim_{n\to \infty}x_{n} = 0.
$$ 

But how fast does $x_n$ converge to $0$? For example, let $x_0=1$. Then the approximate values of $x_n$ for some $n$ are

\[
\begin{array}{|c|c|c|c|c|c|c|c|c|}
\hline
\text{} & x_{300} & x_{400} & x_{500} & x_{600} & x_{700} & x_{800} & x_{900} & x_{1000} \\
\hline
x_n & 0.100 & 0.086 & 0.077 & 0.070 & 0.065 & 0.061 & 0.057 & 0.054 \\
\hline
\end{array}
\]

This suggests power decay: 
$$
x_{n} \sim C n^{-b}
$$ 
for some $C, b>0$. 

Let us assume that this is so and determine $b$ numerically. To this end,  
let us compute $-\frac{1}{m}\log_2 x_{2^m}$, which is supposed to converge to $b$, for a large $m$. For $m=20$ we get the approximate value $0.485$ which suggests that $b=\frac{1}{2}$. 

Assuming this is the case, let us find $C$ numerically. For this let us compute $nx_n^2$, which is supposed to converge to $C^2$, for a large $n$. This sequence varies very slowly (which confirms our guess that $b=\frac{1}{2}$), and for $n=1000$ we get the approximate value $2.98$, which suggests 
that $C^2=3$, i.e., $C=\sqrt{3}$. Thus we expect that 
$$
x_{n} \sim \sqrt{\frac{3}{n}},\ n\to \infty.
$$

\section{Heuristic derivation of the leading asymptotics} Can we now derive this formula analytically? 
To this end let us look for a function $y=f(x)$ such that 
$$
C(n+1)^{-b}=f\left(Cn^{-b}\right),\ n\in \Bbb R_{>0}.
$$
If $x=C n^{-b}$ then $n=\left(\frac{x}{C}\right)^{-\frac{1}{b}}$ and thus
$$
y=C(n+1)^{-b}=Cn^{-b}(1+n^{-1})^{-b}=x\left(1+\left(\frac{x}{C}\right)^{\frac{1}{b}}\right)^{-b}.
$$

Near $x=0$ this expands as
$$
y=x\left(1-b C^{-\frac{1}{b}} x^{\frac{1}{b}}+\cdots\right).
$$

It is natural to expect that the correct values of $b,C$ are the ones for which this expansion is the best possible approximation of $\sin x$ near $x=0$.
Recall that
$$
\sin x=x-\frac{x^{3}}{6}+\cdots=x\left(1-\frac{x^{2}}{6}+\cdots\right).
$$
Thus for the best fit we need 
$
\frac{1}{b}=2$,\ $\frac{1}{2} C^{-2}=\frac{1}{6},
$ 
i.e., 
$$
b=\frac{1}{2},\quad 
C=\sqrt{3},
$$ 
which is consistent with the numerical results of the previous section. 

\section{Rigorous proof of the leading asymptotics} But this is a heuristic argument, not a proof. How to turn it into a proof?
For this purpose note the following lemma. Suppose that $f,g: [0, \varepsilon)\to \Bbb R$ are 
continuous functions such that 
$$
0<f(x) \leqslant g(x)<x
$$ 
for $x \in(0, \varepsilon)$ (so $f(0)=g(0)=0$), and $f$ or $g$ is nondecreasing. Let $x_{n}$ solve the recursion 
$$
x_{n+1}=f\left(x_{n}\right)
$$ 
and $y_{n}$ solve the recursion 
$$
y_{n+1}=g\left(y_{n}\right).
$$

\begin{lemma}\label{l1} If $x_0 \leqslant y_0$ then 
$x_n \leqslant y_n$
for all $n \geqslant 0$.
\end{lemma}

\begin{proof} 
The proof is by induction on $n$ starting with $n=0$. 
If $x_{m}\le y_{m}$ then if $g$ is nondecreasing then 
$$
x_{m+1}=f(x_{m})\le g(x_{m})\le g(y_{m})=y_{m+1},
$$
and if $f$ is nondecreasing then 
$$
x_{m+1}=f(x_{m})\le f(y_{m})\le g(y_{m})=y_{m+1}.
$$
\end{proof} 

\begin{cor}\label{c2} For any $x_0, y_0$, there exists $k$ such that
$x_{n+k}\le y_n$ for all $n\ge 0$. 
\end{cor} 

\begin{proof} Since $0<x_{n+1}<x_n$, the sequence $x_n$ tends to some $x$ such that $f(x)=x$, so $x_n\to 0$. Thus for some $k$, $x_k \leqslant y_0$. So the statement follows from Lemma \ref{l1}.
\end{proof} 

Now we can use Corollary \ref{c2} to sandwich $x_{n}$ between two sequences which have known asymptotics. In fact, this works more generally.

\begin{thm}\label{t3} Let $\alpha>0$ and $f: [0,\varepsilon)\to \Bbb R$ be a continuous function such that
$$
f(x)=x-\alpha x^{r+1}+O\left(x^{2 r+1}\right),\ x \rightarrow 0.
$$

Let a sequence $\left\{x_{n}\right\}$ be defined by the recursion
$$
x_{n+1}=f\left(x_{n}\right)
$$
for small enough $x_{0}>0$.
Then for any $0<\delta<1$ we have
$$
x_n=Cn^{-b}(1+O(n^{-\delta})),\ n\to \infty,\ \text {\rm where } b=\frac{1}{r},\ C=(\alpha r)^{-\frac{1}{r}}.
$$
\end{thm} 

\begin{proof} If $y=x-\alpha x^{r+1}+O(x^{2 r+1})$ then 
$$
y^r=(x-\alpha x^{r+1}+O(x^{2 r+1}))^r=x^r-r\alpha x^{2r}+O(x^{3r}).
$$ 
Thus by change of variable $x\mapsto x^{r}$ we can arrange that $r=1$. So without loss of generality we may assume that 
$f(x)=x-\alpha x^{2}+O\left(x^{3}\right)$.
Furthermore, by change of variable $x\mapsto \alpha x$ we can make $\alpha=1$, so we may assume that 
$$
f(x)=x-x^{2}+O\left(x^{3}\right),\ x\to 0.
$$ 
Now we need to show that 
 $x_{n}=\frac{1}{n}(1+O(n^{-\delta}))$, $n\to \infty$.
To this end, fix $\delta<1$ close to $1$ and for $\lambda\in [-1,1]$ 
define the sequence $a_n=a_n(\lambda)$ by 
$$
a_{n}:=\frac{1}{n}\left(1+\lambda n^{-\delta}\right).
$$

We claim that there exists a unique germ\footnote{Let $X\subset \Bbb R$, $x_0\in X$, and $Y$ be a set. Say that two $Y$-valued functions $g_1$, $g_2$ defined on some neighborhoods of $x_0$ in $X$ are equivalent if $g_1(x)=g_2(x)$ for $x$ sufficiently close to $x_0$. An equivalence class in this sense is called a {\bf germ} of a $Y$-valued function on $X$ near $x_0$. We will slightly abuse terminology by also using the word ``germ" for a representative of such an equivalence class.} 
$g_\lambda$ of a continuously differentiable function on $\Bbb R_{\ge 0}$ near $0$ such that 
\begin{equation}\label{glam}
g_\lambda\left(\frac{1}{u}\left(1+\lambda u^{-\delta}\right)\right)=\frac{1}{u+1}\left(1+\lambda (u+1)^{-\delta}\right)
\end{equation}
for $u\gg 0$; in particular, $a_{n+1}=g_\lambda\left(a_{n}\right)$ for large $n$. 
Indeed, if we set $v:=u^{-1}$ then \eqref{glam} takes the form 
$$
g_\lambda\left(v\left(1+\lambda v^{\delta}\right)\right)=\tfrac{v}{1+v}\left(1+\lambda (\tfrac{v}{1+v})^{\delta}\right).
$$
As the function 
$v\mapsto v\left(1+\lambda v^{\delta}\right)$ is continuously differentiable for $v\ge 0$ and has derivative $1$ at $v=0$, the claim follows from the inverse function theorem.

Let us compute the beginning of the asymptotic expansion of $g_\lambda(x)$ as $x\to 0$.
We will work modulo $o(x^{2+\delta})$, omitting all negligible terms, and using that $\delta$ is sufficiently close to $1$.
Setting 
$x:=a_{n}$, 
we have
$$
n=\frac{1+\lambda n^{-\delta}}{x}=
\frac{1+\lambda x^\delta (1+\lambda n^{-\delta})^{-\delta}}{x}=
x^{-1}(1+\lambda x^{\delta}-\delta \lambda^{2} x^{2 \delta}+O(x^{3\delta})), 
$$
so setting $y:=a_{n+1}$, we have
$$
y=\frac{1+\lambda(n+1)^{-\delta}}{n+1}= 
 \frac{x}{1+\lambda x^{\delta}-\delta \lambda^{2} x^{2 \delta}+x+O(x^{3\delta})}\left(1+\frac{\lambda x^{\delta}}{\left(1+\lambda x^{\delta}+x+O(x^{2\delta})\right)^{\delta}}\right)  = 
 $$
 $$
  \frac{x}{1+\lambda x^{\delta}-\delta \lambda^{2} x^{2 \delta}+x+O(x^{3\delta})}\left(1+\lambda x^{\delta}-\delta\lambda^2 x^{2\delta}-\delta \lambda x^{1+\delta}+O(x^{3\delta})\right)=
 $$
 $$
  \frac{x}{1+\frac{x}{1+\lambda x^{\delta}-\delta \lambda^{2} x^{2 \delta}}+O(x^{3\delta})}\left(1-\frac{\delta \lambda x^{1+\delta}}{1+\lambda x^{\delta}-\delta \lambda^{2} x^{2 \delta}}+O(x^{3\delta})\right)=
 $$
 $$
 x\left(1-x+\lambda(1-\delta) x^{1+\delta}\right)+o(x^{2+\delta}),\ x\to 0,
$$
since $1+3\delta>2+\delta$. Thus 
$$
g_{\lambda}(x)=x-x^{2}+\lambda(1-\delta) x^{2+\delta}+o(x^{2+\delta}),\ x\to 0.
$$
Now we see that for small $x>0$ we have 
$$
g_{-1}(x) \leqslant f(x) \leqslant g_{1}(x).
$$ 
So by Corollary \ref{c2} we get that for some $k$ and every sufficiently large $n$
\begin{equation}\label{eqqq}
\frac{1}{n+k}\left(1-(n+k)^{-\delta}\right) \leqslant x_{n} \leqslant \frac{1}{n-k}\left(1+(n-k)^{-\delta}\right) 
\end{equation}
which implies the required statement. 
\end{proof} 

In particular, this applies to $f(x)=\sin x$ since 
$$\sin x=x-\frac{x^{3}}{6}+O\left(x^{5}\right),
$$
so we have now proved rigorously that in this case indeed 
$$
x_n\sim \sqrt{\frac{3}{n}}.
$$

\section{The next term of asymptotics} The next question is, what about a more precise asymptotics? I.e., how fast does $\sqrt{\frac{n}{3}} x_{n}$ approach $1$ (for $f(x)=\sin x$)?  By Theorem \ref{t3},
$\sqrt{\frac{n}{3}} x_{n}=1+O(n^{-\delta})$ for any $\delta<1$, so we may hope that $\sqrt{\frac{n}{3}} x_{n}$ has an expansion in powers of $\frac{1}{n}$, i.e. 
$$\quad x_{n}=\sqrt{\frac{3}{n}}\left(1+\frac{\gamma}{n}+o\left(\frac{1}{n}\right)\right), n \rightarrow \infty.
$$ 
So what is $\gamma$ ? We can take 
$$
a_{n}=\sqrt{\frac{3}{n}}\left(1+\frac{\gamma}{n}\right)
$$ 
and use the above
trick again -- find a function $g_{\gamma}$ such that 
$$
a_{n+1}=g_{\gamma}\left(a_{n}\right)
$$ 
and see for which $\gamma=\gamma_{0}$ it is the best approximation of $\sin x$; namely, we need $g_{\gamma}(x) \leqslant \sin x$ if $\gamma<\gamma_{0}$ and $g_{\gamma}(x) \geqslant \sin x$ if $\gamma>\gamma_{0}$. But we have a bad news: 
$\gamma_{0}=-\infty$, i.e., $g_{\gamma}(x) \geqslant \sin x$ for all $\gamma \in \mathbb{R}$.
So there is {\bf no asymptotic expansion that we'd hoped for!}

In fact what happens is that the next term of the asymptotics is 
$-\frac{3}{10} \frac{\log n}{n}.$ 
Namely, we have

\begin{prop}\label{p4} (\cite{B}, p.159) For $f(x)=\sin x$,
$$
x_{n}=\sqrt{\frac{3}{n}}\left(1-\frac{3}{10} \frac{\log n}{n}+O\left(\frac{1}{n}\right)\right),\ n \rightarrow \infty .
$$
\end{prop}

One can compute a deeper expansion, but the next term and many (although not all) lower order terms depend on $x_{0}$. The full asymptotics is computed in \cite{BR}.

This is a special case of the following more general theorem, which is proved by a method similar to the above. Note that we need the next Taylor coefficient!

\begin{thm}\label{t5} (see e.g. \cite[Corollary 3.4]{BN})\footnote{In the journal version of \cite[Corollary 3.4]{BN},  there is a misprint: in formula (25), $\ell$ in the denominator of the fraction inside parentheses should be replaced by $\ell^2$. It is corrected in the arXiv version. Also the error term in \cite{BN} is a bit coarser, namely $o(\frac{\log n}{n})$, but with a bit more work the same method yields the finer bound $O(\frac{1}{n})$.} If 
$$
f(x)=x-\alpha x^{r+1}+\beta x^{2 r+1}+\cdots,
$$ 
$\alpha>0$, $\beta \in \mathbb{R}$, then the sequence $x_{n}$ defined by $x_{n+1}=f\left(x_{n}\right)$ with small enough $x_{0}>0$
satisfies the asymptotics
$$
x_{n}=(\alpha r n)^{-\frac{1}{r}}\left(1+\left(\frac{\beta}{\alpha^{2}}-\frac{r+1}{2}\right) \frac{1}{r^{2}} \cdot \frac{\log n}{n}+O\left(\frac{1}{n}\right)\right),\ n\to \infty.
$$
\end{thm}

\begin{example} For 
$$
f(x)=\sin x=x-\frac{x^{3}}{6}+\frac{x^{5}}{120}+...
$$
we have 
$$
r=2,\ \alpha=\frac{1}{6},\ \beta=\frac{1}{120},
$$
so we get 
$$
x_{n}=\sqrt{\frac{3}{n}}\left(1-\frac{3}{10} \frac{\log n}{n}+O\left(\frac{1}{n}\right)\right),\ 
n \rightarrow \infty,
$$ 
as claimed in Proposition \ref{p4}.
\end{example} 

\begin{exercise} Find the expansion near $x=0$ 
of the function $h_\beta(x)$ for which the function $b(u):=\frac{1}{u}(1+\frac{\beta}{u}\log u)$ satisfies the equation $b(u+1)=h_\beta(b(u))$ for $u\gg 0$. 
Then prove Theorem \ref{t5} by mimicking the proof of Theorem \ref{t3}. 
\end{exercise} 

\section{Asymptotics and conjugacy classes} So what happens when we iterate a general function $f(x)$ with $f(0)=0$, $f'(0)=1$, $f(x)<x$ for $x>0$? To understand this, we need to think about the problem more algebraically.

Let $\bold G$ be the group of germs of smooth changes of variable $x\mapsto f(x)$ for small $x\ge 0$ (so $f(0)=0$), with $f'(0)=1$ and the group operation being the composition, $(f\circ h)(x)=f(h(x))$. Let $f_1,f_2\in \bold G$.
Suppose that there exists $g\in \bold G$ such that 
$$
f_{2}=g^{-1} \circ f_{1} \circ g.
$$
Then if $\left\{x_{n}\right\}$ solves the recursion $x_{n+1}=f_{2}\left(x_{n}\right)$ then $\left\{g\left(x_{n}\right)\right\}$ solves the recursion $y_{n+1}=f_{1}\left(y_{n}\right)$, and vice versa. So the asymptotics of iterations of $f_{1}$ determines that of $f_{2}$ and vice versa. Thus our problem reduces to classification of conjugacy classes in $\bold G$. 

Let us say that an element $f\in \bold G$ is {\it flat} if 
its Taylor series coincides with that of the identity (which is just $x$), i.e., if for any $j\ge 2$ one has $f^{(j)}(0)=0$; in this case, the graph of $f$ near $0$ is very close to the line $y=x$, hence the terminology. It is clear that flat elements form a normal subgroup $\bold G_{\rm flat}\subset \bold G$. The classification of non-flat conjugacy classes in $\bold G$ is given by the following theorem
of Takens (1973). 

\begin{thm}\label{t6} (\cite{T}, Theorem 2) Non-flat conjugacy classes in $\bold G$ are represented by elements 
$$
f_{r,\alpha,\beta}(x)=x-\alpha x^{r+1}+\beta x^{2 r+1}, \quad \alpha, \beta \in \mathbb{R},\ \alpha \neq 0,\ r\in \Bbb Z_{>0},
$$
and $r, \alpha, \beta$ are uniquely determined.
\end{thm} 

Theorem \ref{t6} implies that the asymptotics of $x_{n}$ for any non-flat\footnote{The case of flat $f$ is more complicated and we will not discuss it.}  $f$ is determined by Theorem \ref{t5}.

In fact, Theorem \ref{t6} has an algebraic version which is valid over any field $\bold k$ of characteristic zero. This version concerns the group $G(\bold k)$
of {\bf formal} changes of variable
$$
f(x)=x+a_{2} x^{2}+a_{3} x^{3}+\cdots,\ a_j\in \bold k 
$$ 
(i.e., with no convergence conditions). 

\begin{thm}\label{t7} (N. Venkov, see \cite{V}, Theorem 1.1) 
Nontrivial conjugacy classes in $G(\bold k)$ are represented by elements 
$f_{r,\alpha,\beta}$ where $\alpha,\beta\in \bold k,\alpha\ne 0$, 
$r\in \Bbb Z_{>0}$, where $r, \alpha, \beta$ are uniquely determined.\footnote{Note that the classification of conjugacy classes of germs of (real or complex) {\it analytic} changes of variable $x\mapsto f(x)$ with $f(0)=0$, $f'(0)=1$ is much more intricate, 
and they are parametrized by infinitely many parameters, see \cite{E},\cite{V}.}
\end{thm} 

Note that by \'E. Borel's lemma (1895; see \cite{GG}, Lemma IV.2.5), any power series with real coefficients is the Taylor series of a smooth function. Thus the Taylor expansion map defines a group isomorphism $\bold G/\bold G_{\rm flat}\cong G(\Bbb R)$. So Theorem \ref{t7} over $\bold k=\Bbb R$ is a corollary of Theorem \ref{t6} (which is however used in its proof in \cite{T}, so must be proved independently). For the same reason Theorem \ref{t7} and Borel's lemma imply the following weaker version of Theorem \ref{t6}:

\begin{theorem}\label{t6a} A non-flat element $f$ of $\bold G$ can be brought by conjugation to the normal form $x-\alpha x^{r+1}+\beta x^{2 r+1}+\varepsilon(x)$, where all the derivatives of $\varepsilon(x)$ at $x=0$ vanish. 
\end{theorem} 

This weaker form is however sufficient for computing the asymptotics of 
the sequence $x_n$ defined by $x_{n+1}=f(x_n)$. 

\begin{proof} (of Theorem \ref{t6}) The element $g$ conjugating $f$ to the normal form $f_{r,\alpha,\beta}$ is constructed by the method of successive approximations, as follows. We have 
$$
f(x)=x-\alpha x^{r+1}+O(x^{r+2})
$$
for some $\alpha\ne 0$ and $r$. It is easy to see that $r,\alpha$ cannot be changed by conjugation. 
But let us attempt to use conjugations to inductively kill the higher coefficients (of degrees $\ge r+2$). So take
$$
h(x)=x-\alpha x^{r+1}+\beta x^{r+s}+\cdots, \quad s>1,\ g(x)=x+c x^{s}. 
$$
Then $g^{-1}(x)=x-cx^s+...$, so
\scriptsize
\begin{equation}\label{esste}
g^{-1} \circ h \circ g(x)=x+c x^{s}-\alpha\left(x+c x^{s}\right)^{r+1}+\beta\left(x+c x^{s}\right)^{r+s} 
 -c\left(x+c x^{s}-\alpha\left(x+c x^{s}\right)^{r+1}+\beta\left(x+c x^{s}\right)^{r+s}\right)^{s}+...
\end{equation}
\normalsize

Note that except the first term $x$, all the terms in this series must have a factor $\alpha$ or $\beta$ (as if $\alpha=\beta=0$ then $g^{-1}\circ h\circ g=h=x$). 
Thus \eqref{esste} contains all the relevant terms to compute 
the expansion of $g^{-1}\circ h\circ g$ modulo $x^{r+s+1}$: 
$$
g^{-1} \circ h \circ g(x)=x-\alpha x^{r+1}+(\beta-c \alpha(r+1-s)) x^{r+s}+\cdots
$$
So if $s \neq r+1$, we can kill the term of degree $s+r$ by letting $\beta=c \alpha(r+1-s)$, i.e. $c=\frac{\beta}{\alpha(r+1-s)}$. 
However, once we have killed all the terms of $f$ with $s\le r$, the term $\beta x^{2r+1}$ corresponding to $s=r+1$ cannot be changed, it is a conjugacy invariant of $f$. 
On the other hand, we can kill all terms with $s>r+1$. This proves Theorem \ref{t7}. 
\end{proof}

{\bf Remark.} To prove Theorem \ref{t6}, one needs to carefully organize the successive approximation process in the proof of Theorem \ref{t6a} so that it converges in the $C^\infty$ sense near $0$. So Theorem \ref{t6} is significantly deeper than Theorem \ref{t6a}. 

\begin{exercise} Let $f(x)=x-x^3+x^4$. Find the first two terms of the asymptotics of the sequence $x_n$ 
such that $x_{n+1}=f(x_n)$ (with sufficiently small $x_0>0$). 
To this end, find $r,\alpha,\beta$ such that $f_{r,\alpha,\beta}$ is conjugate to $f$ in $\bold G$.  
\end{exercise} 

\section{Generalization to positive characteristic} Theorem \ref{t7} does not hold if ${\rm char}(\bold k)=p>0$, since to kill a term of degree $r+s$, we need $s\ne r+1$ modulo $p$, not just in $\Bbb Z$. However, 
we still have a version of this theorem for the truncated 
group $G_p(\bold k)=G(\bold k)/H_p(\bold k)$ 
where $H_p(\bold k)\subset G(\bold k)$ is the normal subgroup 
of elements $f(x)=x+O(x^{p+3}).$ The group $G_p(\bold k)$ consists of polynomials
$$
f(x)=x+a_{2} x^{2}+\cdots+a_{p+2} x^{p+2}
$$ 
with $a_{i} \in \bold k$ under composition modulo terms $x^j$ with $j\ge p+3$.

\begin{thm}\label{t7a} Nontrivial conjugacy classes in $G_p(\bold k)$ are represented by elements 
$f_{r,\alpha,\beta}$ where $\alpha,\beta\in \bold k,\alpha\ne 0$, 
$1\le r\le p+1$, where $r, \alpha, \beta$ are uniquely determined if $r\le \frac{p+1}{2}$, 
otherwise only $\alpha$ is uniquely determined and the conjugacy class of $f_{r,\alpha,\beta}$ is independent of $\beta$. 
\end{thm} 

\begin{proof} This follows immediately from Theorem \ref{t7} since the equality 
$s=r+1$ in $\Bbb Z$ for $r+s\le p+1$ is equivalent to the same equality in $\bold k$. 
\end{proof} 

\section{Application to finite groups} It is interesting that Theorem \ref{t7a} (for $\bold k$ a finite field) can be applied to the asymptotic theory of finite groups. Namely, let $q(N)$ be the smallest number of conjugacy classes of a finite group $\Gamma$ of order $\ge N$. How does $q(N)$ behave as $N$ grows? 

First of all, we have the following simple theorem:

\begin{thm}\label{t8} (E. Landau, 1903) 
$$
\lim_{N\to \infty}q(N)=\infty.
$$ 
\end{thm}

\begin{proof} Suppose a group $\Gamma$ of order $N$ has $k$ conjugacy classes, and let $m_{1}, . .., m_{k}$ be the orders of their centralizers, with $m_1=N$ (as the centralizer of $1\in \Gamma$ is the entire group $\Gamma$). Then the sizes of the conjugacy classes are $\frac{N}{m_{1}}, \ldots, \frac{N}{m_{k}}$, so the class equation yields
$
\frac{N}{m_{1}}+\cdots+\frac{N}{m_{k}}=N \Rightarrow \frac{1}{m_{1}}+\cdots+\frac{1}{m_{k}}=1.
$
But for fixed $k$ and any $r\in \Bbb Q_{>0}$ the equation
$
\frac{1}{m_{1}}+\cdots+\frac{1}{m_{k}}=r
$ 
has only finitely many integer solutions $N=m_1\ge m_2\ge...\ge m_k\ge 1$. This is easy to see by induction: for $k=n$ we have $\frac{1}{m_n}\ge \frac{r}{n}$, so 
$m_n\le \frac{n}{r},$
which gives finitely many options for $m_n$, and the rest of $m_j$ are then determined from the equation 
 $
 \frac{1}{m_{1}}+\cdots+\frac{1}{m_{n-1}}=r-\frac{1}{m_n},
 $
 which has finitely many solutions by the inductive assumption. 

So for every $k$ there are only finitely many possible values of $m_1=N$, 
and the theorem follows. 
\end{proof} 

So how fast does $q(N)$ go to infinity? A lower bound is provided by the following theorem.

\begin{thm}\label{t9} (Pyber, 1992, \cite{P}) There exists $C>0$ such that for $N\gg 0$
$$
q(N) \geqslant C\frac{\log N}{(\log \log N)^8}.
$$
\end{thm} 
The exponent $8$ for $\log\log N$ was improved by Keller to $7$ in 2008 (\cite{K}) and  by Baumeister, Mar\'oti and Tong-Viet to $3+\varepsilon$ for any $\varepsilon>0$ in 2015 (\cite{BMT}). It 
is conjectured that for some $C>0$
$$
q(N) \geqslant C\log N,
$$
but this conjecture is still open, and the exact asymptotics of $q(N)$ is unknown. 

However, we'll now show that Theorem \ref{t9} is, in fact, rather close to a sharp bound.

\begin{thm}\label{t10} (Kovacs and Leedham-Green, 1986, \cite{KLG})
$$
q(N) = O(\log^3 N),\ N\to \infty.
$$
\end{thm} 

\begin{proof} The proof is direct and similar to the one in \cite{KLG}. Let $p$ be a prime and $\Gamma=G_p(\Bbb F_p)$. Then 
by Theorem \ref{t7a}, nontrivial conjugacy classes in $\Gamma$ have representatives
$$
f_{r,\alpha,\beta}(x)=x-\alpha x^{r+1}+\beta x^{2 r+1},\
\alpha, \beta \in \mathbb{F}_p,\
\alpha \neq 0,\ 1 \leqslant r \leqslant p+1.
$$

The number of such representatives is $p(p-1)(p+1)< p^3$. On the other hand, 
$$
|\Gamma|=p^{p+1}.
$$ 
Thus if $N\le p^{p+1}$, i.e., 
\begin{equation}\label{ine}
\log N\le (p+1)\log p,
\end{equation} 
then $q(N)\le p^3$. 
In particular, this holds for $p=p_N$, where 
$p_N$ is the smallest prime satisfying \eqref{ine}, which satisfies the asymptotics 
$p_N \sim \frac{\log N}{\log \log N}.$
Thus $q(N)=O(\log^3 N)$, as claimed.
\end{proof}

\end{document}